\documentclass[11pt, a4paper, reqno]{amsart}

\usepackage{amsfonts,amsmath,amsthm}
\usepackage{amssymb}
\usepackage{latexsym}

\usepackage[english]{babel}

\numberwithin{equation}{section}

\newtheorem{theorem}{Theorem}[section]
\newtheorem{proposition}[theorem]{Proposition}
\newtheorem{lemma}[theorem]{Lemma}
\newtheorem{corollary}[theorem]{Corollary}

\theoremstyle{definition}

\newtheorem{definition}[theorem]{Definition}
\newtheorem{example}[theorem]{Example}

\theoremstyle{remark}

\newtheorem{remark}[theorem]{Remark}

\newcommand{\eps}{\varepsilon}

\newcommand{\R}{{\mathbb R}}

\newcommand{\Id}{\mathrm{Id}}

\newcommand{\dopu}{{:}\allowbreak\ }

\newcommand{\loglike}[1]{\mathop{\rm #1}\nolimits}

\newcommand{\sign}{\loglike{sign}}
\newcommand{\co}{\loglike{co}}

\newcounter{abc}   
\newcounter{iiiii} 

\newenvironment{aequivalenz}
{\setcounter{iiiii}{0}
\begin{list}%
{{\rm (\roman{iiiii})}}
{\usecounter{iiiii}
\parsep=0pt plus 1pt
\topsep=1pt plus 2pt minus 1pt
\itemsep=1pt plus 2pt minus 1pt
\leftmargin=3\baselineskip \labelsep=.6\baselineskip
\labelwidth=2.4\baselineskip
\rightmargin 0pt}%
}
{\end{list}}

\newenvironment{statements}%
{\setcounter{abc}{0}
\begin{list}%
{{\rm (\alph{abc})}}
{\usecounter{abc}
\parsep=0pt plus 1pt
\topsep=1pt plus 2pt minus 1pt
\itemsep=1pt plus 2pt minus 1pt
\leftmargin=3\baselineskip \labelsep=.6\baselineskip
\labelwidth=2.4\baselineskip
\rightmargin 0pt}%
}
{\end{list}}

{\begin{list}%
{$\bullet$}
{\leftmargin=3\baselineskip
\labelsep=\baselineskip \labelwidth=2.5\baselineskip
\parsep=0pt plus 1pt
\topsep=1pt plus 2pt minus 1pt
\itemsep=1pt plus 2pt minus 1pt
\rightmargin 0pt}%
}
{\end{list}}

\newcommand{\bea}{\begin{eqnarray*}}
\newcommand{\eea}{\end{eqnarray*}}
\newcommand{\beq}{\begin{equation}}
\newcommand{\eeq}{\end{equation}}
\newcommand{\begsta}{\begin{statements}}
\def\endsta{\end{statements}}
\newcommand{\begaeq}{\begin{aequivalenz}}
\def\endaeq{\end{aequivalenz}}

\begin{document}

\title[The Daugavet property for
  bilinear maps]{Slice continuity for operators and the Daugavet property for
  bilinear maps} 

\author{Enrique A. S\'anchez P\'erez and Dirk Werner}

\subjclass[2000]{Primary 46B04; secondary  46B25}

\keywords{Daugavet property, slice, bilinear maps}

\thanks{The authors acknowledge with thanks the
support of the Ministerio de  Econom\'{\i}a y Competitividad
(Spain) under the research project MTM2009-14483-c02-02.}

\address{Instituto
Universitario de Matem\'atica Pura y Aplicada,  Universidad
\linebreak
Polit\'ecnica de Valencia, Camino de Vera s/n, 46022 Valencia, Spain.}
\email{easancpe@mat.upv.es}

\address{Department of Mathematics, Freie Universit\"at Berlin,
Arnimallee~6, \qquad {}\linebreak D-14\,195~Berlin, Germany}
\email{werner@math.fu-berlin.de}

\begin{abstract}
We introduce and analyse the notion of slice continuity between
operators on Banach spaces in the setting of the Daugavet
property.
 It is shown that under the slice continuity assumption
the Daugavet equation holds for weakly compact operators. As an
application we define and characterise the Daugavet property for
bilinear maps, and we prove that this allows us to describe some
$p$-convexifications  of the Daugavet equation for operators on
Banach function spaces that have  recently been introduced.
\end{abstract}

\maketitle

\thispagestyle{empty}

\section{Introduction}

A Banach space $X$ is said to satisfy the Daugavet property if 
the so-called Daugavet equation
$$
\| \Id + R \| = 1 + \|R\|
$$
is satisfied for every rank one operator $R\dopu X \to X$. In recent
years, the Daugavet property for Banach spaces has been studied by
several authors, and various applications have been found (see for
instance
\cite{bulpolis2008,ams2000,studia2001,jfa1997,Dirk-IrBull}).

The aim of this paper is to introduce and analyse the notion of
slice continuity between operators on Banach spaces. We will show
that under this assumption one can  easily characterise when the
Daugavet equation holds for a couple of operators $T$ and $R$
between Banach spaces, i.e., when
$$
\| T + R\| = \|T\| + \|R\|.
$$

Recently, some new ideas have been introduced in this direction.
The notion of Daugavet centre has been studied in
\cite{bosen,bosen2,boseka}. According to Definition~1.2 in
\cite{boseka}, a nonzero operator $T$ between (maybe different) Banach
spaces is a Daugavet centre if the above Daugavet equation holds for
every rank one operator~$R$. In this paper we develop a notion that is
in a sense connected to this one but provides a direct tool for
analysing when a particular couple of operators satisfies the
Daugavet equation. Our idea is to relate the set of slices defined
by each of the two operators. Recall that the slice  $S(x',
\varepsilon)$  of the unit ball of a real Banach space $X$ determined by a
norm one element 
$x' \in X'$ and an $ \varepsilon >0$ is the set
$$
S(x', \varepsilon)= \{ x \in B_X \dopu  \langle x,
x' \rangle  \ge  1- \varepsilon  \}.
$$

 Let $Y$ be a Banach space. Let $T\dopu X \to Y$ be an operator. We
will define the set of slices associated to $T$ by
$$
S_T:=\{S(T'(y')/\|T'(y')\|,\varepsilon) \dopu  0< \varepsilon < 1, \ y'
\in Y', \  T'(y') \ne 0 \},
$$
and we will say that an operator $R\dopu X \to Y$ is slice continuous
with respect to $T$
-- we will write $S_R \le S_T$ -- if for every $S \in S_R$ there is a
slice $S_1 \in S_T$ such that $S_1 \subset S$. This notion will
be used for characterising when the Daugavet equation holds by
adapting
 some of the known results on the geometric
description of the Daugavet property to our setting. From the
technical point of view, we use some arguments on the Daugavet
property defined by subspaces of $X'$ that can be found in
\cite{bulpolis2008}. This is done in Section~2. In Section~3 we
develop the framework for using our results in the setting of the
bilinear maps in order to obtain the main results of the paper
regarding applications. Finally in Section~4 we provide examples
and applications, mainly related to a unified general point of
view to understand the $p$-convexification of the Daugavet
equation for Banach function spaces that have recently been
studied in \cite{EnrDir2,EnrDir3}.

Our notation is standard.
 Let $X,Y$ and $Z$ be real Banach spaces. $B_X$ and $S_X$ are the unit
 ball and the unit sphere of $X$, respectively. We write $U_X$ for the
 open unit ball and $X'$ for the dual space. We denote by 
$L(X,Y)$ the space of continuous operators and by $B(X \times Y, Z)$
the space of continuous bilinear maps from $X \times Y$ to $Z$. If
$T$ is an operator, we write $T'$ for its adjoint operator. If $x'
\in X'$ and $y \in Y$, we identify the tensor $x' \otimes y$ with
the operator $x' \otimes y\dopu X \to Y$ mapping $x$ to
$x'(x)y$. Throughout the paper all the 
bilinear maps are assumed to be continuous. In general, we
consider the norm $\|(x,y)\|=\max\{\|x\|_X, \|y\|_Y \}$ for the
direct product $X \times Y$. We will say that a bilinear map $B$
is convex or has convex range (resp.\ is weakly compact) if the norm
closure of $B(B_X,B_Y)$ is convex (resp.\ weakly compact).

Regarding Banach function spaces we also use standard notation. If
$1 \le p \le \infty$ we write $p'$ for the extended real number
satisfying
 $1/p+1/p'=1$. Let us fix some definitions and basic results.
  Let $(\Omega, \Sigma, \mu)$ be a measure space.
A Banach function space $X(\mu)$ over the measure $\mu$ is an
order ideal of $L^0(\mu)$ (the space of $\mu$-a.e.\ equivalence
classes of integrable functions)  that is a Banach
space with a lattice norm $\| \,.\, \|$ such that for every $A \in
\Sigma$ of finite measure, $\chi_A \in X(\mu)$ (see
\cite[Def.~1.b.17]{lint}). We will write $X$ instead of $X(\mu)$ if
the 
measure $\mu$ is clear from the context.
  Of course, Banach function spaces are Banach lattices, so the
following definition makes sense for these spaces. Let $0 < p \le
\infty$. A Banach lattice $E$ is $p$-convex if there is a constant
$K$ such that for each finite sequence $(x_i)_{i=1}^n$ in $E$,
$$
\Big\| \Big( \sum_{i = 1}^n |x_i|^p \Big)^{1/p} \Big\|_E
 \le K \Big(\sum_{i=1}^n \|x_i\|_E^p \Big)^{1/p}.
$$
It is said that it is $p$-concave if there is a constant $k$ such
that for every finite sequence $(x_i)_{i=1}^n$ in $X$,
$$
\Big(\sum_{i=1}^n \|x_i\|_E^p \Big)^{1/p} \le  k \Big\| \Big(
\sum_{i=1}^n |x_i|^p \Big)^{1/p} \Big\|_E.
$$
$M^{(p)}(E)$ and $M_{(p)}(E)$ are the best constants in these
inequalities, respectively.

Let $0 \le p < \infty$. Consider  a Banach function space
$X(\mu)$. Then the set
$$
X(\mu)_{[p]}:=\{h \in L^0(\mu)\dopu  |h|^{1/p} \in X(\mu) \}
$$
is called the $p$-th power of $X(\mu)$, which is a quasi-Banach
function space when endowed with the quasinorm
$\|h\|_{X_{[p]}}:=\| |h|^{1/p} \|^p_X$, $h \in X^p$ (see
\cite{defant}, \cite{maligranda-persson,CDS} or
\cite[Ch.~2]{libro}; the symbols that are used there for this
concept are $X^p$, $X^{1/p}$ and $X_{[p]}$, respectively); if $X$
is $p$-convex and $M^{(p)}(X(\mu))=1$ -- we will say that $X$ is
constant~$1$ $p$-convex --, then $X(\mu)_{[p]}$ is a Banach
function space, since in this case $\|\,.\,\|_{X_{[p]}}$ is a
norm; if $0<p<1$, the $p$-th power of a Banach function space is
always a Banach function space. Every $p$-convex Banach lattice
can be renormed in such a way that the new norm is a lattice norm
with $p$-convexity constant~$1$ (\cite[Prop.~1.d.8]{lint}). Let $f
\in X$. Throughout the paper we use the notation $f^p$ for the sign
preserving $p$-th power of the function $f$, i.e., $f^p:= \sign\{f\}
|f|^p$.

\begin{remark} \label{rembasic}
The following basic facts regarding $p$-th powers of Banach
function spaces will be used several times. Their proofs are
immediate using the results in \cite[Ch.~2]{libro}. Let $X(\mu)$ be a
Banach function space and $0<p<\infty$.
\begin{itemize}
\item[(a)]  For every couple of functions $f \in X$ and $g \in
X_{[p/p']}$ one has $\|fg \|_{X_{[p]}} \allowbreak \le \|f\|_X \|g\|_{X_{[p/p']}}$,
and
$$
\|h\|_{X_{[p]}}= \inf \{ \|f\|_X \|g\|_{X_{[p/p']}}\dopu   fg=h, \ f
\in X, \ g \in X_{[p/p']} \}.
$$

\item[(b)] For every $h \in X_{[p]}$ one has  $h= |h|^{1/p} h^{1/p'}$,
$h^{1/p} \in X$, $h^{1/p'} \in X_{[p/p']}$, and
\bea
\|h\|_{X_{[p]}} &=& \| h^{1/p}\|_X^p = \|h^{1/p} \|_X
\|h^{1/p}\|_X^{p/p'} \\
&=& \|h^{1/p} \|_X \|(h^{1/p'})^{p'/p}\|_X^{p/p'}=\|h^{1/p} \|_X
\|h^{1/p'} \|_{X_{[p/p']}}.
\eea
\end{itemize}
\end{remark}

\section{Slice continuity for couples of linear maps}

Let us start by adapting  some facts that are already essentially
well known (see \cite{ams2000}).

\begin{proposition} \label{geom0}
Let $X$ and $Y$ be Banach spaces. Let $T\dopu X \to Y$ be a norm one
linear map, and consider a norm one linear form $x' \in X'$. Let
$y \in Y \setminus \{0\}$. The following assertions are
equivalent.
\begin{itemize}

\item[(1)] $ \| T+ x' \otimes y\|=1 + \|x' \otimes y\|= 1 + \|y\|.$

\item[(2)]
For every $\varepsilon
>0$ there is an element $x \in S(x',\varepsilon)$ such that
$$
\Bigl\|  T(x) + \frac{y}{\|y\|} \Bigr\| \ge 2- 2 \varepsilon.
$$
\end{itemize}
\end{proposition}
\begin{proof}

(1)$\Rightarrow$(2). By Lemma~11.4 in \cite{abraali0} (or
\cite[p.~78]{Dirk-IrBull}) we can assume that $y \in S_Y$. By
hypothesis, $\|T+ x' \otimes y\|= 1 + \|y\|=2$, and then there is
an element $x \in B_X$ such that
$$
2-  \varepsilon \le \| T(x)+ \langle x, x' \rangle y)\| \le
\|T(x)\|+|\langle x,x' \rangle| \le 1 + | \langle x,x' \rangle|.
$$
Note that we can assume that $\langle x,x' \rangle >0$;  otherwise
 take $-x$ instead of $x$. Since for every $\eps>0$
\bea
2- \varepsilon &\le& \|T(x) + \langle x, x' \rangle y\| \le \|T(x)
+y\| + \|\langle x, x' \rangle y -y\| \\
&\le& \|T(x) +y\| + (1- \langle x, x' \rangle)\|y\| \le \|T(x) + y\|
+ \varepsilon,
\eea
we obtain~(2).

(2)$\Rightarrow$(1). Let $x' \in S_{X'}$ and $y \in Y$ and
consider the rank one  map $x' \otimes y$.  Again by Lemma~11.4 in
\cite{abraali0} we need consider only the case $\|y\|=1$. Let $
\varepsilon
>0$. Then there is an $x \in S(x',\varepsilon)$ such that
$\|y+ T(x)\| \ge 2- 2\varepsilon.$ Thus,
\bea
2- 2\varepsilon &\le& \| y + T(x)\| \le \|y- \langle x,x' \rangle
y\| + \|\langle x,x' \rangle y+ T(x)\| \\
&\le& (1- \langle x,x' \rangle)\|y\| + \|\langle x, x' \rangle y+
T(x)\| \le \varepsilon + \|x' \otimes y +T\|.
\eea
Consequently, $\|x' \otimes y\|+ \|T\| = 2 =\|x' \otimes y +T\|$.
\end{proof}

When a subset of linear maps $V \subset L(X,Y)$ is considered,
the following generalisation of the Daugavet property makes sense.

\begin{definition}
Let $X,Y$ be Banach spaces and let $T\dopu X \to Y$ be a norm one
operator. The Banach space $Y$ has the $T$-Daugavet property with
respect to $V \subset L(X,Y)$ if for every $R \in V$,
$$
\|T + R\|=1 + \|R\|.
$$
\end{definition}

This definition encompasses the notion of Daugavet centre given in
Definition~1.2 of \cite{boseka}.

\begin{corollary} \label{geom20}
Let $X$ and $Y$ be Banach spaces. Let $T\dopu X \to Y$ be an operator,
and consider a set of norm one linear forms $W \subset X'$. Let
$W \cdot Y=\{x' \otimes y\dopu  x' \in W,\ y \in Y \}$. The following
statements are equivalent.
\begin{itemize}
\item[(1)]
$Y$ has the $T$-Daugavet property with respect to $W \cdot Y$.

\item[(2)]
For every $y \in S_Y$, for every $x' \in W$ and for every
$\varepsilon >0$ there is an element $x \in S(x',\varepsilon)$
such that
$$
\| T(x)+y\| \ge 2- 2 \varepsilon.
$$

\end{itemize}
\end{corollary}

\begin{definition}
Let $T\dopu X \to Y $ be a continuous linear map. Let $y' \in Y'$.
 We denote by $T_{y'}\dopu  X \to \mathbb R $ the linear
form given by $T_{y'}(x):= \frac{ \langle x,T'(y')
\rangle}{\|T'(y')\|}$ whenever $T'(y') \ne 0$. The natural set of
slices defined by $T$ is then
$$
S_T=\{S({T_{y'}}, \varepsilon)\dopu  0<\varepsilon < 1, \ y' \in Y',
\ T'(y') \ne 0 \}.
$$
If  $R\dopu X \to Y$ is another operator, we use the symbol $S_{R} \le
S_{T}$ to denote that for every slice $S$ in $S_{R}$ there is a
slice $S_1 \in S_{T}$ such that $S_1 \subset S$. We will say in
this case that $R$ is slice continuous with respect to $T$.
\end{definition}

For  operators $T$ having particular properties, slice continuity
allows easy geometric descriptions. Let $T\dopu X \to Y $ be an
operator between Banach spaces such that $T'$  is an isometry onto
its range, i.e., $T$ is a quotient map, and let $R\dopu X \to Y$ be an
operator. The following assertions are equivalent.
\begin{itemize}
\item[(1)]
 $S_R \le S_{T}.$
\item[(2)]
For every $y \in S_Y$, $y' \in S_{Y'}$ such that $R'(y') \ne 0$,
and every $\varepsilon
>0$ there is an element $y_0' \in S_{Y'}$  such that
$ (R_{y'} \otimes y)(S(T'(y_0'),\varepsilon)) \subset
B_\varepsilon(y). $
\end{itemize}

To see this just  notice that for every $y' \in S_{Y'}$ such that
$R'(y') \ne 0$ and $y \in S_Y$
$$
S( R_{y'},\varepsilon)= \{x \in B_X\dopu  1-\varepsilon \le R_{y'}(x)
\le 1\} = \{x \in B_X\dopu  \|R_{y'}(x)y-y\| \le \varepsilon \}.
$$

For a general operator $T$ the canonical example of when the
relation $S_{R} \le S_{T}$ holds is given by the case
 $R=P \circ T$, where $T\dopu X \to Y$  and $P\dopu Y \to Y$ are
operators. In this case, $\langle R(x), y' \rangle= \langle x,
T'(P'(y')) \rangle$, and so clearly $S_R \le S_{T}$. So the reason
is that we have the inclusion $R'(Y') \subset T'(Y')$. However,
there are examples of couples of operators $T,R$ such that $R$ is
slice continuous with respect to $T$ but $R \ne P \circ T$ for any
operator $P$. Let us show one of them.

\begin{example}
Let $T\dopu  C[0,1]\oplus_1 \mathbb R \to C[0,1]$, $T(f,\alpha)=f$, and
$R\dopu  C[0,1] \oplus_1 \mathbb R \to C[0,1]$, $R(f,\alpha)=f+\alpha
\mathbf 1$, where $\mathbf{1}$ stands for the constant one
function and $\oplus_1 $ denotes the direct sum with
the $1$-norm. Then $R$ and $T$ have norm one. Since the kernel of
$T$ is not contained in the kernel of $R$, we do not have $R= P
\circ T$ for any operator $P$. But the slice condition holds. A
simple calculation gives that for every $\mu$ in the unit sphere
of $C[0,1]^*$, $\|T'(\mu)\|=\|R'(\mu)\|=1$. Let $S_r \in S_R$ be
the slice generated by any $\mu\in C[0,1]^*$ of norm one and
$\varepsilon>0$. We claim that the slice $S_t$ generated by the
same $\mu$ and $\varepsilon/2$ is contained in $S_r$. Indeed, if
$(f,\alpha)$ is in the unit ball and $\langle\mu, f\rangle \ge
1-\varepsilon/2$, then $\|f\|\ge 1-\varepsilon/2$ and hence
$|\alpha | \le \varepsilon/2$. Therefore, for such $(f,\alpha)$,
$\langle\mu, f+\alpha\mathbf 1 \rangle \ge \langle\mu, f\rangle -
|\alpha| \ge 1-\varepsilon$, and so the inclusion $S_t \subset
S_r$ holds.

\end{example}

\begin{remark} \label{re10}
 Let $T,R\dopu X
 \to Y$ be a couple of operators, $\|T\|=1$.  Notice that
 Proposition~\ref{geom0} gives that for every $y \in S_Y$  and $y' \in   
Y'$ such that $R'(y') \ne 0$, the following are equivalent.
\begin{itemize}

\item[(1)] $ \| T+ R_{y'} \otimes y\|=2.$

\item[(2)]
For every $\varepsilon>0$ there is an element $x \in S(
R_{y'},\varepsilon)$ such that
$$
\|  T(x) + y \| \ge 2- 2 \varepsilon.
$$
\end{itemize}

Thus for the case $R=T$ and assuming that $T'$ is an isometry onto
its range we obtain that $Y$ has the Daugavet property if and only
if $Y$ has the $T$-Daugavet property with respect to the set
$\{T_{y'}\dopu y' \in Y' \setminus \{ 0\} \} \cdot Y$. This is  a
direct consequence of the well-known characterisation of the
Daugavet property (see Lemma~2.1 in \cite{ams2000}) and
Corollary~\ref{geom20}. Consequently, for any other $R$, if $Y$ has
the 
Daugavet property and $S_R \le S_{T}$, we obtain that for every $y
\in S_Y$ and every $y' \in Y'$ such that $R'(y') \ne 0$,
$$
\|T + R_{y'} \otimes y\|=2.
$$

Note that something like  the slice continuity requirement $S_R \le S_{T}$ is
necessary for this to be true; indeed, a quotient map
$T\dopu X \to Y$ is not necessarily a Daugavet centre, even if the
spaces involved have the Daugavet property. Take the operators
$T,R\dopu L^1[0,1] \oplus_1 L^1[1,2] \to L^1[0,1]$ given by
$T((f,g)):=f$ and $R((f,g)):=(\int_1^2 g \,dx) \cdot h_0$, $(f,g)
\in L^1[0,1] \oplus_1 L^1[1,2]$, where 
$h_0$ is a norm one function in $L^1[0,1]$.
Clearly, $\|T\|=\|R\|=1$, but $\|T+R\| \le 1$.
\end{remark}

\begin{theorem} \label{weakcom0}
Let $Y$ be a Banach space with the Daugavet property. Let $T\dopu X \to
Y$ be an operator such that $T'$ is an isometry onto its range and
$R\dopu X \to Y$ a norm one operator. Then:
\begin{itemize}
\item[(1)]
 If for every $\varepsilon >0$ there is a
slice $S_0 \in S_{T}$ and an element $y \in S_Y$ such that $R(S_0)
\subset B_\varepsilon(y)$, then
$$
\| T + R \|= 2.
$$
\item[(2)]
If $S_R \le S_{T}$ and $R$ is weakly compact, then
$$
 \|T  + R\|= 2.
$$
\end{itemize}
\end{theorem}
\begin{proof}
(1) Take $\varepsilon >0$. Then there are $S_0=S(T_{y'_0},\delta)
\in S_{T}$ and $y \in S_Y$ such that for every $x \in S_0$, $\|
R(x)- y \| \le \varepsilon$. We can assume that $\delta \le
\varepsilon$. Since $Y$ has the Daugavet property, $Y$ has the
$T$-Daugavet property with respect to the set $\{ T_{y'}\dopu  y' \in
Y' \setminus \{0\} \} \cdot Y $ (see Remark \ref{re10} above).
Therefore, by Corollary~\ref{geom20} we find an element $x \in
S_0$ such that
$$
\|T+R\| \ge \| T(x)+ y\| - \|y-R(x)\| \ge 2- \varepsilon -2 \delta
\ge 2- 3 \varepsilon.
$$
Since this holds for every $\varepsilon >0$, the proof of (1) is
complete.

The proof of (2) follows the same argument as the one for
operators in spaces with the Daugavet property (see
\cite[Th.~2.3]{ams2000}), so we only sketch it. Assume that
$\|R\|=1$. Since 
by hypothesis $K=\overline{R(B_X )}$ is a convex weakly compact
set, it is the closed convex hull of its strongly exposed points.
Since this set is convex and $\|R\|=1$, there is a strongly
exposed point $y_0 \in K$ such that $\|y_0\| \le 1$ and $\|y_0\| >
1- \varepsilon$. Take a functional $y'_0$ that strongly exposes
$y_0$ and satisfies $\langle y_0,y_0'\rangle= \max_{y \in K}
\langle y,y_0'\rangle =1$. It can be proved by contradiction that
there is a slice $S \in S_R$ such that  $R(S)$ is contained in the
ball $B_\varepsilon(y_0)$ (see the proof of
\cite[Th.~2.3]{ams2000}). Since $S_R \le S_{T} $, there is also a
slice $S_0 
\in S_{T}$ such that $R(S_0) \subset R(S) \subset
B_\varepsilon(y_0)$. Then part (1) gives the result.
\end{proof}

The example in Remark~\ref{re10} makes it clear that some condition
like slice continuity
is necessary for (2) in Theorem \ref{weakcom0} to be true. The
following variation of this example gives a genuine weakly compact
operator that is not 
of finite rank which does not satisfy the Daugavet equation.
Take $T$ defined as in Remark~\ref{re10} and $R\dopu L^1[0,1] \oplus_1
L^2[1,2] \to L^1[0,1]$ given by $R((f,g)):=g(x-1)$. This operator is weakly
compact and $\|R\|=\|T\|=1$, but again the norm of the sum of both
operators is less than $2$.

\begin{remark} \label{themainrem}
Notice that the condition in (1) on the existence of a slice $S
\in S_T$ such that $R(S) \subset B_\varepsilon(y)$ can be
substituted by the existence of a slice $S \in S_T$ and a $\delta
>0$ such that $R(S+ \delta B_X) \subset B_\varepsilon(y)$. The
argument given in the proof based on this fact makes it also  clear 
that the relation $S_T \le S_R$ can be substituted by the
following weaker one and the result is still true: For every slice
$S \in S_R$ and $\delta >0$ there is a slice $S_1 \in S_T$ such
that
$$
S_1  \subset S +  \delta B_X.
$$
\end{remark}

\section{Bilinear maps and the Daugavet property}

In this section we analyse the Daugavet property for bilinear maps
defined on Banach spaces. Our main idea is to provide a framework
for the understanding of several new Daugavet type properties and
prove some general versions of the main theorems that hold for the
case of the Daugavet property. We centre our attention on the
extension of the Daugavet equation for weakly compact bilinear
maps. Let $X,Y$ and $Z$ be Banach spaces. Consider a norm one
continuous bilinear map $B\dopu X \times Y \to Z$. Then we can consider
the linearisation $T_B\dopu X \hat{\otimes}_\pi Y \to Z$, where $X
\hat{\otimes}_\pi Y$ is the projective tensor product with the
projective norm $\pi$ (see for instance \cite[Sec.~3.2]{deflo} or
\cite[Th.~2.9]{Ryan}). 
This linear
operator will provide meaningful results for bilinear maps by
applying the ones of Section~2. However, a genuinely geometric
setting for bilinear operators -- slices, isometric
equations, \dots\ -- will  also be defined  in this section
in order to provide the specific links between the (bilinear)
slice continuity and the Daugavet equation.

We will consider bilinear operators $B_0\dopu  X \times Y \to Z$
satisfying that $B_0(U_X \times U_Y)=U_Z$. Obviously,  such a
map has always convex range, i.e., $B_0(U_X \times U_Y)$ is a
convex set. We will say that a map satisfying these conditions is
a norming bilinear map. If $B_0$ is such a bilinear operator, we
will say that a Banach space $Z$ has the $B_0$-Daugavet property
with respect to the class of bilinear maps $V \subset B(X \times
Y, Z)$ if
$$
\|B_0+B\|= 1 + \|B\|
$$
for all $B \in V$. Notice that  $Z$ has the $B_0$-Daugavet
property  with respect to $V$ if and only if it has the
$T_{B_0}$-Daugavet property with respect to the set $\{T_B\dopu  X
\hat{\otimes}_\pi Y \to Z\dopu B \in V\}$. Let us consider some examples.

\begin{example} \label{primex}
(1)
Take a Banach space $X$ and consider the bilinear form $B_0\dopu X
\times X' \to \mathbb{R}$ given by $B_0(x,x')=\langle x, x'
\rangle$, $x \in X$, $x' \in X'$. Consider the set
\bea
\lefteqn{V= \{B_T\dopu X \times X'\to\R\dopu 
B_T(x,x')=\langle T(x), x' \rangle, } \hspace{4cm} \\
&&T\dopu X \to X \ \textrm{is weakly compact} \}.
\eea
Then notice that
$$
\sup_{x \in B_X, x' \in B_{X'}} | B_0(x,x')+B_T(x,x') |  = \sup_{x
\in B_X, x' \in B_{X'}} | \langle x+T(x),x' \rangle |= \|\Id + T\|
$$
and $\|B_0\|+\|B_T\|= 1+ \|T\|$. Therefore $\mathbb R$ has the
$B_0$-Daugavet property with respect to $V$ if and only if $X$ has
the Daugavet property (see Theorem~2.3 in \cite{ams2000}).

(2)
Take a measure space $(\Omega,\Sigma,\mu)$ and a couple of Banach
function spaces $X(\mu)=X$ and $Z(\mu)=Z$ over $\mu$ satisfying
that the space of multiplication operators $X^Z$ is a saturated
Banach function space over $\mu$ and $X$ is $Z$-perfect, i.e.,
$(X^Z)^Z=X$, and $U_{X} \cdot U_{X^Z}=U_Z$ (here $\cdot$ represents the
pointwise product of functions). Consider the bilinear map $B_0\dopu X
\times X^Z \to Z$ given by $B_0(f,g)= f \cdot g$, $f \in X$, $g
\in X^Z$ (see \cite{CDS} for definitions and results regarding
multiplication operators on Banach function spaces). Consider the
set
\bea
\lefteqn{V= \{B_S\dopu X \times X^Z \to Z \dopu   B_S(f,g)= S(f \cdot
  g),} 
\hspace{4cm} \\
&& S\dopu Z \to Z \ \textrm{is weakly compact} \}.
\eea
Then
$$
\sup_{f \in B_X, \ g \in B_{X^Z}} \| B_0(f,g)+B_S(f,g) \|_Z
= \sup_{f \in B_X, \ g \in B_{X^Z}} \| f \cdot g +S(f\cdot g) \|_Z=
\|\Id + S\|
$$
and $\|B_0\|+\|B_S\|= 1+ \|S\|$. Therefore $Z$ has the
$B_0$-Daugavet property with respect to $V$ if and only if $Z$ has
the Daugavet property (see again Theorem~2.3 in \cite{ams2000}).

(3)
Take $1 < p < \infty$, its conjugate index $p'$, a measurable
space $(\Omega, \Sigma)$, a Banach space $Z$ and a countably
additive vector measure $m\dopu  \Sigma \to Z$.  Consider the
corresponding spaces of $m$-integrable functions $L^p(m)$ and
$L^{p'}(m)$, and the bilinear map $B_0\dopu  L^p(m)\times L^{p'}(m) \to
Z$ given by the composition of the multiplication and the
integration map $I_m\dopu L^1(m) \to Z$, i.e., $B_0(f,g)= \int fg \, dm
$. This map is well defined and continuous (see
\cite[Chapter~3]{libro} for the main definitions and results on the
spaces 
$L^p(m)$). Assume also that $B_0(U_{L^p(m)}\times U_{L^{p'}(m)})=
I_m(U_{L^p(m)} \cdot U_{L^{p'}(m)})$ coincides with the open unit
ball of $Z$. Take the set
\bea
\lefteqn{U= \big\{B_R\dopu L^p(m) \times L^{p'}(m) \to Z \dopu  } 
\hspace{2cm}\\
&&  B_R(f,g):= R(I_m(f \cdot g)), \
R\dopu Z \to Z \ \textrm{rank one} \big\}.
\eea
Since
$$
\|B_0 + B_R\| = \sup_{f \in B_{L^p(m)},\ g \in B_{L^{p'}(m)}} \Bigl\|
\int_\Omega fg \, dm + R \Bigl(\int_\Omega fg \, dm\Bigr) \Bigr\|_Z 
= \|\Id + R\|
$$
and $\|B_0\|+\|B_R\|= 1+ \|R\|$, we obtain again that $Z$ has the
$B_0$-Daugavet property with respect to $U$ if and only if $Z$ has
the Daugavet property.
\end{example}

\begin{remark}
More examples can be given by considering the following bilinear
maps:

(i) $B_{C(K)}\dopu C(K) \times C(K) \to C(K)$, $B_{C(K)}(f,g)= f \cdot
g$.

(ii) $B_*\dopu L^1(\R) \times L^1(\R) \to L^1(\R)$, $B_*(f,g)=
f * g$, where $*$ is the convolution product. In this case we
have $B_*(U_{L^1(\R)},U_{L^1(\R)})=U_{L^1(\R)}$ as a consequence of
 Cohen's Factorisation Theorem (see Corollary~32.30 in
\cite{hewross}).

(iii) For a $\sigma$-finite $\mu$,  $B_{L^\infty}\dopu  L^\infty(\mu)
\times L^1(\mu) \to \mathbb{R}$ given by $B_{L^\infty}(f,g)= \int
fg \, d \mu$.

\end{remark}

Bilinear operators for which the Daugavet equation will be shown
to hold -- together with norming bilinear maps -- are weakly
compact operators with convex range. Although the usual way of
finding such a map is to compose a bilinear map with convex range
and a weakly compact linear one, other examples can be given. Let
us show one of them that is in fact  not norming.

\begin{example}
Consider a constant~$1$ $p$-convex reflexive Banach function space
$X$. In particular, $X$ must be order continuous. Take $f'_0 \in
S_{X'}$ and $f_0 \in S_X$ and define the bilinear map $B\dopu X \times
X_{[p/p']} \to X_{[p]}$ given by $B(f,g)= \langle f, f'_0 \rangle
f_0 \cdot g$. Note that $\|B\|=1$. Let us show that the (norm)
closure $K=\overline{B(B_X \times B_{X_{[p/p']}})}$ is a convex
weakly compact set.

Let $z_1,z_2 \in B(B_X \times B_{X_{[p/p']}})$. Let $f_1,g_1,f_2$
and $g_2$ such that $B(f_1,g_1)=z_1$ and
 $B(f_2,g_2)=z_2$. Take $0 < \alpha <
1$ and consider the element $\alpha z_1+(1-\alpha) z_2$. Let us
prove that it belongs to $B(B_X \times B_{X_{[p/p']}})$. Notice
that since $-1 \le \langle f,f_0' \rangle \le 1$ for every $f \in
B_X$, $g_3=\alpha \langle f_1,f'_0 \rangle g_1 +(1-\alpha) \langle
f_2, f'_0 \rangle g_2$ belongs to $B_{X_{[p/p']}}$. Take now an
element $f_3 \in B_X$ such that $\langle f_3, f'_0 \rangle =1$ (it
exists since $X$ is reflexive), and note that
$$
B(f_3,g_3)= \alpha z_1 + (1-\alpha) z_2.
$$
So, $K$ is convex. Notice that ${B(B_X \times B_{X_{[p/p']}})}$ is
also relatively weakly compact; it is enough to observe that the
set is uniformly $\mu$-absolutely continuous (see for instance
Remark~2.38 in \cite{libro} and the references therein), i.e., that
$$
\lim_{\mu(A) \to 0} \sup_{z \in K} \|z \chi_A\|=0.
$$
But this is a direct consequence of the fact that $X$ is order
continuous (see for instance \cite[Th.~1.c.5 and Prop.~1.a.8]{lint})
and the H\"older inequality for the norms of $p$-th 
power spaces (adapt \cite[Lemma~2.21]{libro} or
\cite[Prop.~1.d.2(i)]{lint}). For every $z = \langle f, f'_0 \rangle f_0
\cdot g \in B(B_X,B_{X_{[p/p']}})$ and $A \in \Sigma$,
$$
\| z \chi_A\|_{X_{[p]}} = \| \langle f, f'_0 \rangle f_0 \cdot
g\|_{X_{[p]}} \le |\langle f, f'_0 \rangle| \|f_0 \chi_A\|_X \cdot
\|g\|_{X_{[p/p']}}.
$$
Since $X$ is order continuous, $\|f_0 \chi_A\|_X \to 0$ when
$\mu(A) \to 0$, which  gives the result.

\end{example}

Let us now start  to adapt the results of the previous section. In
order to do so, let us define the natural set of slices associated
to a norm one bilinear form $b \in B(X \times Y,\mathbb R)$. Let
$0<\varepsilon <1$. Following the notation given for the linear
case, we define $S(b,\varepsilon)$ by
$$
S(b,\varepsilon):= \{(x,y)\dopu  x \in B_X,\ y \in B_Y,\ b(x,y) \ge
1-\varepsilon \}.
$$
The following result shows the relation between slices defined by
a bilinear form and the ones defined
  by the linearisation of this map.

\begin{lemma} \label{lemacompa}
Let
$b \in B(X \times Y, \mathbb R)$ be a norm one bilinear form (i.e.,
$T_b \in (X \hat{\otimes}_\pi Y)'$ with norm one) and $\varepsilon
>0$. Then: 
\begin{itemize}

\item[(1)]
There is an  elementary tensor $x \otimes y$ such that $\|x\|=\|y\|=1$
and $x \otimes y \in S(T_b,\varepsilon)$. 

\item[(2)] $\overline{\co \{x \otimes y\dopu  (x,y) \in
    S(b,\varepsilon) \}} \subset S(T_b,\varepsilon).$ 

\item[(3)] $S(T_b,\varepsilon^2) \subset  \overline{\co \{x \otimes
    y\dopu  (x,y) \in S(b,\varepsilon) \}} + 4 \varepsilon B_{X
    \hat{\otimes}_\pi Y}.$ 
\end{itemize}
\end{lemma}
\begin{proof}
(1) Take a norm one element $t \in S(T_b, \varepsilon/2)$. Then there
is an element $t_0=\sum_{i=1}^n \alpha_i x_i \otimes y_i \in X \otimes
Y$ such that $\|x_i\|=\|y_i\|=1$, $\alpha_i >0$ for all $i=1,\dots,n$,
$\sum_{i=1}^n \alpha_i=1$ and  $\pi(t - t_0) < \varepsilon/2$. Then 
$$
\langle t_0, T_b \rangle = \langle t-t_0, T_b \rangle + \langle t, T_b
\rangle \ge \langle t, T_b \rangle -|\langle t-t_0, T_b \rangle| > 1-
\frac{\varepsilon}{2}-\frac{\varepsilon}{2}, 
$$
and so $t_0 \in S(T_b, \varepsilon)$. Assume (by changing the signs of
some of the $x_i$ if necessary) that 
$b(x_i,y_i) >0$ for all $i$. Then
$$
\sum_{i=1}^n \alpha_i b(x_i,y_i) \ge \sum_{i=1}^n \alpha_i (1- \varepsilon),
$$
and so there is at least one index $i_0$ such that $b(x_{i_0},
y_{i_0}) \ge 1- \varepsilon$. Consequently, $x_{i_0} \otimes y_{i_0}
\in S(T_b,\varepsilon)$. 

(2) is a direct consequence of the fact that $S(T_b,\varepsilon)$ is
norm closed in the projective tensor product.

(3) Let us show now that  $S(T_b,\varepsilon^2) \subset
\overline{\co \{x \otimes y\dopu  (x,y) \in S(b,\varepsilon) \}} +
4\varepsilon B_{X \hat{\otimes}_\pi Y}$. Let $u\in
S(T_b,\varepsilon^2)$. Find $v$ such that $\|v\|<1$, $T_b(v)\ge
1-\varepsilon^2$, and $\|v-u\|\le\varepsilon$. Write
$v=\sum_{i=1}^\infty \alpha_i x_i \otimes y_i$ with all
$\|x_i\|=\|y_i\| =1$, $\alpha_i\ge0$ and
$\alpha:=\sum_{i=1}^\infty \alpha_i <1$. Note that
$\alpha\ge1-\varepsilon^2$. Now consider
\begin{eqnarray*}
I & := & \{i \in \mathbb N\dopu  b(x_i,y_i) \ge 1-\varepsilon\}
=  \{i\in \mathbb N\dopu  (x_i,y_i)\in S(b,\varepsilon)\}, \\
J & := & \{i\in \mathbb N\dopu  b(x_i,y_i) < 1-\varepsilon\}.
\end{eqnarray*}
Let $\alpha_I:= \sum_{i\in I} \alpha_i$ and  $\alpha_J:=
\sum_{i\in J} \alpha_i$. We have
$$
1-\varepsilon^2 \le \sum_{i=1}^\infty \alpha_i b(x_i,y_i) \le
\alpha_I + \alpha_J (1-\varepsilon) < 1-\varepsilon\alpha_J
$$
and hence $\alpha_J<\varepsilon$. Let $w=\sum_I
\frac{\alpha_i}{\alpha_I} x_i \otimes y_i \in  \overline{\co \{x
\otimes y\dopu  (x,y) \in S(b,\varepsilon) \}} $; we then have (note
that $v=\alpha_I w + \sum_J \alpha_i x_i \otimes y_i$)
$$
\|v-w\| \le |\alpha_I-1| \|w\| + \alpha_J.
$$
Furthermore $0\le 1-\alpha_I = \alpha_J + 1 -\alpha \le
\varepsilon + \varepsilon^2$; hence
$$
\|u-w\| \le \|u-v\| + \|v-w\| \le \varepsilon + ((\varepsilon +
\varepsilon^2) +\varepsilon^2) \le 4\varepsilon,
$$
as claimed.
\end{proof}

If $z \in Z$, we define $b_z\dopu X \times Y \to Z$ as the (rank one)
bilinear map given by $b_z(x,y)=b(x,y) z$, $x \in X$, $y \in Y$.
Let $B\dopu X \times Y \to Z$ be a continuous bilinear map. In what
follows we need to introduce some elements related to duality and
adjoint bilinear operators. Following Ramanujan and Schock in
\cite{ramsch}, we consider the adjoint operator $B^\times\dopu Z' \to
B(X,Y)$ given by $B^\times (z')(x,y)= \langle B(x,y), z' \rangle$
(this definition does not coincide with the one given originally by
Arens in \cite{arens}, although the setting is of course the
same). $B^\times$ is a linear and continuous operator, and
$\|B\|=\|B^\times\|$.

\begin{definition}
Let $B\dopu X \times Y \to Z$ be a continuous bilinear map. Let $z' \in
S_{Z'}$ and consider the adjoint bilinear form $\langle B,
z'\rangle\dopu X \times Y \to \mathbb{R}$ given by $\langle B,z'
\rangle(x,y)=B^\times (z')(x,y)$.
 We denote by $B_{z'}\dopu  X \times Y \to \mathbb{R}$ the bilinear
form given by $B_{z'}(x,y)= \frac{\langle B(x,y), z'
\rangle}{\|\langle B, z' \rangle\|}$ whenever $\|\langle B, z'
\rangle\| \ne 0$ and by $\langle B, Z' \rangle$ the set of all
these bilinear forms. The natural set of slices defined by $B$ is
then
$$
S_B=\{S(B_{z'}, \varepsilon)\dopu  0<\varepsilon < 1, \ \|B_{z'}\| \ne
0 \}.
$$
If $B_1$ is another (continuous) bilinear map, $B_1\dopu X \times Y \to
Z$, we use the symbol $S_{B} \le S_{B_1}$ to denote that for every
slice $S$ in $S_{B}$ there is a slice $S_1 \in S_{B_1}$ such that
$S_1 \subset S$. We can also consider  the relation $S_{T_B} \le
S_{B_1}$ to be defined in the same way: for every $S \in S_{T_B}$
there is a slice $S_1 \in S_{B_1}$ such that the set $\{x \otimes
y\dopu  (x,y) \in S_1 \}$ is included in $S$. Lemma~\ref{lemacompa}
gives an idea of how this relation works. 
\end{definition}

As in the linear case, the canonical example of the relation $S_{B}
\le S_{B_1}$ between 
sets of slices associated to two bilinear maps is given by
bilinear maps $B$ that are defined as a composition $T \circ B_1$,
where $B_1\dopu X \times Y \to Z$ is a continuous bilinear map and $T\dopu Z
\to Z$ is a continuous operator. In this case, $\langle B(x,y),
z' \rangle= \langle B_1(x,y), T'(z') \rangle$, and so clearly
$S_B \le S_{B_1}$.
Let us show some examples.

\begin{example} \label{exslices}
Let $(\Omega, \Sigma, \mu)$ be a finite measure space and consider
a rearrangement invariant (r.i.)\  constant~$1$ $p$-convex Banach
function space $X(\mu)$ (see \cite[p.~28 and Sections~1.d,
2.e]{lint} or \cite[Ch.~2 and p.~202]{libro}). In this case,
$(X(\mu)_{[p]})'$ is also r.i. Take a measurable bijection
$\Phi\dopu \Omega \to \Omega$ such that $\mu( \Phi(A))=\mu(A)$ for
every $A \in \Sigma$. Then it is possible to define the isometry
$T_r\dopu X_{[r]} \to X_{[r]}$, $0 \le r \le p$, by $T_r(f)= f \circ
\Phi$.

Define the bilinear map $B\dopu X \times X_{[p/p']} \to X_{[p]}$ given
by $B(f,g)= T_1(f) \cdot T_{p/p'}(g)$. Let us show the relation
between the slices defined by $B_0\dopu X \times X_{[p/p']} \to
X_{[p]}$, $B_0(f,g)=fg$, and the slices defined by $B$. Assume
also that $X$ is order continuous. Then $X_{[p]}$ is also order
continuous and the dual of the space can be identified with the
K\"othe dual, which is also r.i., and so every continuous linear
form is an integral. Note that in this case the property $S_B \le
S_{B_0}$ holds, since for every couple of functions $f \in X$ and
$g \in X_{[p/p']}$, $B(f,g)= T_1(f) \cdot T_{p/p'}(g)= (f \circ
\Phi) \cdot (g \circ \Phi)= (f \cdot g) \circ \Phi$. Consequently,
every element $z' \in S_{(X_{[p]})'}$ satisfies that for every
pair of functions $f$ and $g$ as above,
\bea
\langle B_0(f,g), z' \rangle &=& \int_\Omega fg z' \, d\mu= 
\int_\Omega ((fg) \circ \Phi) \cdot  (z' \circ \Phi) \, d\mu\\
&=& \int_\Omega B(f,g) \cdot (z' \circ \Phi) \, d\mu= \langle B(f,g), z'
\circ \Phi \rangle.
\eea
Therefore, there is a one-to-one correspondence between $S_{B_0}$
and $S_B$ given by identifying $S((B_0)_{z'}, \varepsilon)$ and
$S(B_{z' \circ \Phi}, \varepsilon)$, which implies that
$S_{B_0}=S_B$.
\end{example}

Fix a norming bilinear map $B_0\dopu X \times Y \to Z$ and consider a
norm one bilinear map $B\dopu X \times Y \to Z$. Let us provide now
geometric and topological properties for $B$ that imply that the
Daugavet equation is satisfied for $B_0$ and $B$, i.e.,
$\|B_0+B\|=2$. These properties  will be proved as applications of
the result of the previous section.

\begin{corollary} \label{weakcom}
Let $Z$ be a Banach space with the Daugavet property. Let $B_0\dopu X
\times Y \to Z$ be a norming bilinear map and $B\dopu X \times Y
\to Z$ a continuous bilinear map. Then: 

\begin{itemize}
\item[(1)]
 If for every $\varepsilon >0$ there is a
slice $S_0 \in S_{B_0}$ and an element $z \in S_Z$ such that
$B(S_0) \subset B_\varepsilon(z)$, then
$$
\| B_0 + B \|= 1 + \|B\|.
$$
\item[(2)]
If $S_{T_B} \le S_{B_0}$ and $T_B$ is weakly compact
(equivalently, $B(B_X \times B_Y)$ is a relatively weakly compact set), then
$$
\|B_0 + B\|=1 + \|B\|.
$$

\end{itemize}
\end{corollary}
\begin{proof}
(1) is just a consequence of Theorem~\ref{weakcom0}(1): let us
take $\varepsilon/5$ and apply this theorem to $T_{B_0}$ and
$T_B$. By hypothesis there is a slice $S_0=S(b,\delta) \in
S_{B_0}$ such that $B(S_0) \subset B_{\varepsilon/5}(z)$.  We
can assume without loss of generality that $\delta \le
\varepsilon$ and $\|B\| \le 1$. Then by Lemma~\ref{lemacompa}(3),
\bea
T_{B}(S(T_b,\delta^2)) &\subset& \textstyle
T_{B}(\overline{\co(S_0)})+ 
\frac45 \varepsilon T_{B}(B_{X \hat{\otimes}_\pi Y}) \\
&\subset& \textstyle
\overline{\co(B(S_0))} + \frac45  \varepsilon \|T_{B}\|B_Z
\subset B_{\varepsilon/5}(z) + \frac45  \varepsilon B_Z  \\
&\subset&
B_\varepsilon (z),
\eea
and (1) is proved. For (2), apply Theorem~\ref{weakcom0}(2) and
Remark~\ref{themainrem}.

\end{proof}

\begin{example}
It is well known that for a purely non-atomic measure $\mu$ and a
Banach space $E$ the space of Bochner integrable functions
$L^1(\mu,E)$ has the Daugavet property (see \cite{ams2000}). The
next simple application of Corollary~\ref{weakcom} provides a
similar result for the Pettis norm $\|\,.\,\|_P$, i.e., for
operators $T$ from $(L^1(\mu,E),\|\,.\,\|_{L^1(\mu,E)})$ to the
normed space $(L^1(\mu,E),\|\,.\,\|_P)$. Consider the bilinear
map $B_0\dopu L^1(\mu,E) \times E' \to L^1(\mu)$ given by
$$
B_0(f,x')(w)= \langle f(w), x' \rangle, \qquad w \in \Omega.
$$
Take an operator $T\dopu L^1(\mu,E) \to L^1(\mu,E)$ and define the
bilinear map $B_T\dopu L^1(\mu,E) \times E' \to L^1(\mu)$ given by
$$
B_T(f,x')(w)= \langle (T(f))(w), x' \rangle, \qquad w \in \Omega.
$$
Assume that $B_T$ is weakly compact and has convex range and
suppose that $S_{B_T} \le S_{B_0}$ (or that $T_{B_T}$ is weakly
compact and $S_{T_{B_T}} \le S_{T_{B_0}}$). Then by
Corollary~\ref{weakcom}(3) (or~(4)), 
\bea
\sup_{f \in B_{L^1(\mu,E)}} \|f + T(f)\|_P
&=& \sup_{f \in
B_{L^1(\mu,E)}, \ x' \in B_{E'}}
\|B_0(f,x')+B_T(f,x')\|_{L^1(\mu)}\\
&=& \sup_{f \in B_{L^1(\mu,E)}, \ x' \in B_{E'}} \|B_0(f,x')\| \\
&& \mbox{} \qquad {} +
\sup_{f \in B_{L^1(\mu,E)}, \ x' \in B_{E'}}
\|B_T(f,x')\|_{L^1(\mu)} \\
 &=& \sup_{f \in B_{L^1(\mu,E)}} \|f\|_P +  \sup_{f \in
   B_{L^1(\mu,E)}} \|T(f)\|_P .
\eea

\end{example}

Corollary~\ref{weakcom} suggests that the natural examples of
bilinear maps that satisfy the Daugavet equation with respect to
$B_0$ are the ones defined as $B=T \circ B_0$, where $T\dopu Z \to Z$
is a weakly compact operator. Corollary~\ref{coroweakcom}
generalises in a sense the idea of (2) and (3) in
Example~\ref{primex}. Notice, however, that there are other simple 
bilinear maps that fit into the Daugavet setting, as the following
example shows.

\begin{example} \label{nuevo}
Let us show an example of a bilinear map $B\dopu X \times Y \to Z$ such
that $B_0$ and $B$ satisfy the Daugavet equation but there is no
operator $T\dopu Z \to Z$ such that $B=T \circ B_0$. Let $(\Omega,
\Sigma, \mu)$, $X(\mu)$ and $\Phi$ be as in Example~\ref{exslices}
and consider the isometry $T_1\dopu X \to X$ defined there. Assume also
that $\mu(\Omega) < \infty $ and the constant~$1$ function
satisfies $\|\chi_\Omega\|_X=1$. Consider the bilinear map $B\dopu X
\times X_{[p/p']} \to X_{[p]}$ given by $B(f,g)=T _1(f) \cdot g$.
Then, since $T_1(\chi_\Omega)= \chi_\Omega$,
\bea
2 &\ge& \|B_0+B\| =  \sup_{f \in B_X, \ g \in B_{X_{[p/p']}}} \|fg +
T_1(f)g \|_{X_{[p]}} \\
&=&\sup_{f \in B_X, \ g \in B_{X_{[p/p']}}} \|(f + T_1(f))g
\|_{X_{[p]}} =\sup_{f \in B_X} \|f+T_1(f)\|_X \\
 &\ge& \| \chi_\Omega
+ \chi_\Omega\|=2.
\eea 
Notice that in general a bilinear map defined in this way cannot
be written as $T \circ B_0$ for any operator $T$. For instance,
suppose that there is a set $B \in \Sigma$ such that $0 < \mu(B)$
and $B \cap \Phi(B) = \emptyset$ and consider a couple of
non-trivial functions $f_1$ and $f_2$ in $X$ with support in
$\Phi(B)$ and $B$, respectively, and such that $\|(f_1\circ \Phi)
\cdot f_2\| >0$. Then $B_0(f_1,f_2)=0$, but $B(f_1,f_2)\ne 0$, so
there is no operator $T\dopu X_{[p]} \to X_{[p]}$ such that $B= T \circ
B_0$.
\end{example}

\begin{corollary} \label{coroweakcom}
Let $B_0\dopu  X \times Y \to Z$ be a norming bilinear map. Consider
the subsets 
$R$, $C$ and $WC$ of $L(Z,Z)$ of rank one, compact and weakly
compact operators, respectively, and the sets $R \circ B_0=\{B=T
\circ B_0\dopu X \times Y \to Z\dopu   T \in R\}$, $C \circ B_0=\{B=T
\circ B_0\dopu X \times Y \to Z\dopu   T \in C\}$ and $ WC \circ B_0=\{B=T
\circ B_0\dopu X \times Y \to Z\dopu   T \in WC\}$. Then the following are
equivalent.
\begin{itemize}
\item[(1)]$Z$ has the Daugavet property.

\item[(2)] $Z$ has the $B_0$-Daugavet property with respect to $R
\circ B_0$.

\item[(3)] $Z$ has the $B_0$-Daugavet property with respect to $C
\circ B_0$.

\item[(4)] $Z$ has the $B_0$-Daugavet property with respect to $WC
\circ B_0$.

\item[(5)] For every norm one operator $T \in R$, every $z \in Z$
and every  $\varepsilon >0$  there is an element $(x,y) \in S(T
\circ B_0, \varepsilon)$ such that
$$
\|z + B_0(x,y)\| \ge 2- \varepsilon.
$$

\end{itemize}

\end{corollary}
\begin{proof}
The equivalence between (1) and (2) is a direct consequence of the
following equalities. For every rank one operator $T\dopu Z \to Z$,
$$
\|\Id+T\|= \sup_{z \in B_Z} \|z+T(z)\| = \sup_{x \in B_X,y \in B_Y}
\|B_0(x,y)+T(B_0(x,y))\|.
$$
Since the norm closure of the convex hull  $B(B_X \times B_Y)$ is a weakly
compact set, (2) implies~(4) as a consequence of
Corollary~\ref{weakcom}(2). Obviously (4) implies~(2), and so the 
equivalence of (2) and (3) is also clear. The equivalence of (2)
and (5) holds as a direct consequence of Corollary~\ref{geom20} and
the arguments used above. 
\end{proof}

\begin{remark}
Conditions under which a bilinear map $B\dopu X \times Y \to Z$ is
compact or weakly compact (i.e., the norm closure $\overline{B(B_X
\times B_Y)}$ is compact or weakly compact, respectively) have
been studied in several papers; see \cite{ramsch,Ruch} for
compactness and \cite{arongalindo,ulger} for weak compactness. The
reader can find in these papers some factorisation theorems and
other characterisations of these properties, also related with the
notion of Arens regularity of a bilinear map.
\end{remark}

\section{Applications. $p$-convexifications of the Daugavet property and bilinear maps}

Different $p$-convexifications of the Daugavet property have been introduced
in \cite{EnrDir2,EnrDir3}. In this section we show that in a sense
they can be considered as particular cases of a Daugavet property
for bilinear maps. We centre our attention on the case of Banach
function spaces such that their $p$-th powers have the Daugavet
property that have been characterised in \cite{EnrDir2}. However,
more examples of applications will be given as well. Throughout
this section $\mu$ is supposed to be finite.

We explain now two suitable examples of $p$-convexification of the
Daugavet property. Let us start with one regarding $p$-concavity
in Banach function spaces.

\begin{example}
Let $1 \le p < \infty$. Consider a constant~$1$ $p$-convex Banach
function space $X$, $Y=X_{[p/p']} \oplus_\infty X_{[p/p']}$ (the
direct product with the maximum norm), $Z=X_{[p]}$, and the
bilinear map $B_0\dopu  X \times ( X_{[p/p']} \times_\infty X_{[p/p']})
\to X_{[p]}$ given by  $B_0(f,(g,h))= f \cdot P_1(g,h)= fg$.
 Take an operator $T\dopu X_{} \to X_{}$ and consider the
bilinear map $B\dopu  X \times ( X_{[p/p']} \oplus_\infty X_{[p/p']})
\to X_{[p]}$ given by $B(f,(g,h))= T(f) \cdot P_2(g,h)= fh$ (here
$P_1$ and $P_2$ denote the two natural projections in the product
space $X_{[p/p']} \oplus_\infty X_{[p/p']}$). A direct calculation
shows that in this case the Daugavet equation for the pair given
by $B_0$ and $B$ is
$$
\|B_0 + B \| = 1+ \|T\|,
$$
since $\|B\|=\|T\|$. Assume that $\|T\|=1$. Then $\|T(f)\| \le 1$
for every $f \in B_X$, and so, taking $g= f^{p/p'} \in
B_{X_{[p/p']}}$ and $h=T(f)^{p/p'} \in B_{X_{[p/p']}}$ for each $f
\in B_X$, we obtain
\bea
2 \ge \|B_0 + B \| &=& \sup_{f \in B_X, \ g \in B_{X_{[p/p']}},\ h \in
B_{X_{[p/p']}}} \| fg + T(f)\cdot h\|_{X_{[p]}}\\
&\ge& \sup_{f \in B_X} \| |f|^p + |T(f)|^p \|_{X_{[p]}} \\
&\ge& \sup_{f \in B_X} \| (|f|^p + |T(f)|^p)^{1/p}\|^p_{X} .
\eea
Thus, if $X$ is also a constant~$1$ $p$-concave space (i.e., $X$ is
an $L^p$-space) we get
$$
\sup_{f \in B_X} \| (|f|^p + |T(f)|^p)^{1/p}\|^p_{X}
 \ge \sup_{f \in B_X} ( \|f\|_X^p + \|T(f)\|_X^p)= 2.
$$
Therefore, in this case the Daugavet equation holds for $B_0$ and
for every  bilinear map $B$ defined by an operator $T\dopu X_{} \to
X_{}$ in the way explained above.
\end{example}

The following construction shows another example of a Daugavet type
property for a bilinear map that is in fact a $p$-convex version
of the Daugavet property, in the sense that is studied in
\cite{EnrDir3}.

\begin{example} \label{exDaueq?}
Let $(\Omega, \Sigma, \mu)$ be a measure space and consider an
r.i.\  constant~$1$ $p$-convex Banach function space $X(\mu)$.
Consider as in Example~\ref{exslices} the bilinear map $B_0$ given
by the product and a measurable bijection $\Phi\dopu \Omega \to \Omega$
satisfying  that $\mu( \Phi(A))=\mu(A)$ for every $A \in
\Sigma$ and the isometries $T_r\dopu X_{[r]} \to X_{[r]}$, $0 < r \le
p$.

Take the bilinear map $B\dopu X \times X_{[p/p']} \to X_{[p]}$ given by
$B(f,g)= T_1(f) \cdot T_{p/p'}(g)$. Notice that $\|B\|=1$. Then
\bea
2 \ge \|B_0+B\| &\ge& \sup_{f \in B_X, \ g \in B_{X_{[p/p']}}}
\|B_0(f,g)+ B(f,g)\|_{X_{[p]}} \\
&\ge& \sup_{f \in B_X} \| f \cdot f^{p/p'} + T_1(f) \cdot
T_{p/p'}(f^{p/p'}) \|_{X_{[p]}} \\
&=& \sup_{x \in B_X} \|f^p + T_1(f)^p\|_{X_{[p]}} \\
&=& \sup_{f \in B_X} \| |f^p+ T_1(f)^p|^{1/p} \|_X^p   .
\eea
Now, if $\Phi$ satisfies that there is a set $A \in \Sigma$ such
that $\mu(A \cap \Phi(A)) < \mu(A)$, there is a norm one function
$f_0$ such that $f_0$ and $T_1(f_0)$ are disjoint and
$\|T_1(f_0)\|=1$. Assume that $X$ is also $p$-concave (constant~$1$),
i.e., $X$ is an $L^p$-space. Then 
$$
\sup_{f \in B_X} \| |f^p+ T_1(f)^p|^{1/p} \|_X^p 
 \ge \|f_0\|_X^p + \|T_1(f_0)\|_X^p=2,
$$
and thus the so called $p$-Daugavet equation is satisfied for
$T_1$ (see Definition~1.1 in \cite{EnrDir3}), and $B$ and $B_0$
satisfy the Daugavet equation.
\end{example}

Let $1 \le p < \infty$. In what follows we study the $p$-convex
spaces whose $p$-th powers satisfy the Daugavet property by giving
some general results in the setting of the examples presented
above. We analyse the case of $X=X(\mu)$, a constant~$1$
$p$-convex Banach function space, $Y= X(\mu)_{[p/p']}$,
$Z=X(\mu)_{[p]}$, and $B_0\dopu X \times X_{[p/p']} \to X_{[p]}$ given
by $B_0(f,g)=f \cdot g$. We assume that $X_{[p]}$ has the Daugavet
property. The main example we have in mind is given by
$X=L^p[0,1]$, $Y=X_{[p/p']}=L^{p'}[0,1]$ and $Z=X_{[p]}=
L^1[0,1]$. Recall that $\mu$ is assumed to be finite.

\begin{definition}
Let $X(\mu)$, $Y(\mu)$ and $Z(\mu)$ three Banach function spaces
over $\mu$. We say that a continuous bilinear map $B\dopu X(\mu) \times
Y(\mu) \to Z(\mu)$ satisfying that for every $A,C \in \Sigma$,
$B(\chi_A,\chi_C)= B(\chi_{A \cap C},
\chi_{A \cup C})$, is a symmetric bilinear map.
\end{definition}

\begin{proposition} \label{pthDAugavet}
 Let $X(\mu)$ be an order continuous $p$-convex Banach function space 
with $p$-convexity constant equal to~$1$. Then the following assertions are
 equivalent.
\begin{itemize}
\item[(1)]  For every rank one operator $T\dopu X(\mu)_{[p]} \to
X(\mu)_{[p]}$,
 $$
 \sup_{f \in B_X} \||f^p+ T(f^p) |^{1/p} \|^p_X=1 + \|T\|.
 $$

\item[(2)]  For every rank one operator $T\dopu X(\mu)_{[p]} \to
X(\mu)_{[p]}$,
$$
\| B_0 + T \circ B_0\| = 1 + \|T\|.
$$
\item[(3)] For every $z \in S_{X_{[p]}}$, for every  $x' \in
S_{(X_{[p]})'}$ and for every $\varepsilon >0$ there is an
element $(f,g) \in S((B_0)_{x'},\varepsilon)$ such that
$$
\| z + B_0(f,g)\|_{X_{[p]}} \ge 2- 2 \varepsilon.
$$

\item[(4)] Each  weakly compact symmetric bilinear map
$B\dopu X(\mu) \times X_{[p/p']} \to X_{[p]}$ satisfies the equation
$$
\| B_0+B\|=1 + \|B\|.
$$

\item[(5)] $X_{[p]}$ has the Daugavet property.

\end{itemize}
\end{proposition}
\begin{proof}
For the equivalence of (1) and (2), note that the constant~$1$
$p$-convexity of $X$ implies that $B_X \cdot B_{X_{[p/p']}}=
B_{X_{[p]}}$ is the unit ball of the Banach function space
$X_{[p]}$; so, using also Remark~\ref{rembasic} the following
inequalities are obtained:
\bea
\sup_{f \in B_X} \||f^p+ T(f^p) |^{1/p} \|^p_X 
&\le&  \sup_{f \in B_X, \ g \in B_{X_{[p/p']}}} \| fg + T(fg) \|_X \\
&\le& \sup_{h \in B_{X_{[p]}}} \|h + T(h)\|_{X_{[p]}} \\
&\le& \sup_{f \in B_X} \||f^p+ T(f^p) |^{1/p} \|^p_X ,
\eea
and then both assertions are seen to be equivalent. The equivalence of (2)
and (3) is obtained by applying  Corollary~\ref{geom20} to the
setting of bilinear maps.

 Taking into account that the map
$i_{[p]}\dopu X \to X_{[p]}$ given by $i_{[p]}(f)=f^p$ is a bijection
satisfying $\|i_{[p]}(f)\|_{X_{[p]}}=\|f\|^p_X$ for every $f \in
X$, and the definition of the norm $\|\,.\,\|_{X_{[p]}}$, the
equivalence of (1) and (5) is also clear using the well-known
geometric characterisation of the Daugavet property in terms of
slices (see for instance Lemma~2.2 in \cite{ams2000}).

Thus, it only remains to prove the equivalence of (2) and~(4). Let us
show first the following \textit{\textbf{Claim:} Let $X$ be a
$p$-convex (constant~$1$) Banach function space such that the simple
functions are dense and let $B\dopu X(\mu) \times X(\mu)_{[p/p']}
\to X(\mu)_{[p]}$ be a continuous bilinear map. Then $B$ is
symmetric if and only if there is an operator $T\dopu  X_{[p]} \to
X_{[p]}$ such that $B=T \circ B_0$.}

In order to prove this, note that by hypothesis the set $S(\mu)$ of
simple functions is dense in $X(\mu)$ and so for every $0 \le r \le p$ 
it is 
also dense in $X(\mu)_{[r]}$; this can be shown by a direct
computation just considering the definition of the norm in
$\|\,.\,\|_{X_{[r]}}$ and the fact that if $X$ is constant~$1$
$p$-convex then it is constant~$1$ $r$-convex for all such $r$,
see for instance \cite[Prop.~1.b.5]{lint} or
\cite[Prop.~2.54]{libro}. So this holds for $r= p/p'$. If
$B$ is symmetric, then for every couple of simple functions $f=
\sum_{i=1}^n \alpha_i \chi_{A_i}$ and $g = \sum_{j=1}^m \beta_j
\chi_{B_j}$, where $\{A_i\}_{i=1}^n$ and $\{B_i\}_{j=1}^m$ are
sequences of pairwise disjoint measurable sets,
\bea
B(f,g)&=& \sum_{i=1}^n \sum_{j=1}^m \alpha_i \beta_j
B(\chi_{A_i},\chi_{B_j}) \\
&=& \sum_{j=1}^m \sum_{i=1}^n  \alpha_i
\beta_j B(\chi_{A_i \cap B_j},\chi_{A_i \cup B_j})\\
&=& \sum_{j=1}^m \sum_{i=1}^n \beta_j \alpha_i
B(\chi_{B_j},\chi_{A_i}) = B(g,f).
\eea
Therefore, because of the continuity of $B$ and the order
continuity of the spaces, $B(f,g)=B(g,f)$ for every couple of simple
functions $f,g \in X \cap X_{[p/p']}$. Define now the map
$T\dopu X_{[p]} \to X_{[p]}$ by $T(h)=B(f,g)$ for every function
$h=fg$, first for products of simple functions and then
by density for the rest of
the elements of $X_{[p]}$ (note that the norm closure of the set
$(S(\mu) \cap B_X ) \cdot (S(\mu) \cap B_{X_{[p/p']}})$ coincides
with $B_{X_{[p]}}$). It can easily be proved that $T$ is well
defined since $B$ is symmetric. For if $f_1,g_1,f_2,g_2$ are
simple functions with $f_1 g_1=f_2 g_2$, then
$B(f_1,g_1)=B(f_2,g_2)$, and by continuity of $B$, $B(f,g)= B(g,f)$ 
for every couple $f \in X$ and $g \in X_{[p/p']}$. Further,
$T$ is continuous also by the continuity of $B$ and 
Remark~\ref{rembasic}. Consequently, $B=T \circ B_0$ and the claim is
proved.

Thus, (2) is equivalent to (4) as a consequence of
Corollary~\ref{coroweakcom}, since the operator $T$ constructed in the
Claim 
is weakly compact if and only if $B$ is weakly compact.
\end{proof}

\end{document}